\documentclass{elsarticle}
\usepackage{latexsym}
\usepackage{amsfonts}
\usepackage{amsmath, amssymb, amsthm, calrsfs, amsbsy}
\usepackage{subfigure}
\usepackage{graphicx}
\usepackage{multirow}
\usepackage{lscape}

\numberwithin{equation}{section}

\theoremstyle{plain}
\newtheorem{proposition}{Proposition}
\newtheorem{lemma}[proposition]{Lemma}

\newcommand{\eqd}{\stackrel{d}{=}}
\newcommand{\diff}{d}

\newcommand{\defn}{\emph}
\newcommand{\pit}{{\textstyle\frac{\pi}{2}}}
\renewcommand{\Pr}{\operatorname{P}}
\newcommand{\corr}{\operatorname{corr}}

\renewcommand{\le}{\leqslant}
\renewcommand{\ge}{\geqslant}

\renewcommand{\le}{\leqslant}
\renewcommand{\ge}{\geqslant}

\newcommand{\bPhi}{\overline{\Phi}}

\newcommand{\eps}{\varepsilon}

\newcommand{\E}{\operatorname{E}}
\newcommand{\expec}{\operatorname{E}}

\begin{document}
\title{Tails of correlation mixtures of elliptical copulas}

\author[hm]{Hans Manner\corref{cor1}}
\ead{h.manner@maastrichtuniversity.nl} 
\address[hm]{Department of Quantitative Economics, Maastricht University, PO Box 616, NL-6200 MD, Maastricht, The Netherlands.}
\cortext[cor1]{Corresponding author.} 

\author[js]{Johan Segers}
\ead{johan.segers@uclouvain.be}
\address[js]{Institut de statistique, Universit\'e catholique de Louvain, B-1348 Louvain-la-Neuve, Belgium.}

\begin{frontmatter}
 
\begin{abstract}
Correlation mixtures of elliptical copulas arise when the correlation parameter is driven itself by a latent random process. For such copulas, both penultimate and asymptotic tail dependence are much larger than for ordinary elliptical copulas with the same unconditional correlation. Furthermore, for Gaussian and Student \textit{t}-copulas, tail dependence at sub-asymptotic levels is generally larger than in the limit, which can have serious consequences for estimation and evaluation of extreme risk. Finally, although correlation mixtures of Gaussian copulas inherit the property of asymptotic independence, at the same time they fall in the newly defined category of \emph{near asymptotic dependence}. The consequences of these findings for modeling are assessed by means of a simulation study and a case study involving financial time series. \bigskip
\end{abstract}


\end{frontmatter}

\section{Introduction}\label{Sec:intro}
It is a stylized fact that financial data such as stock or exchange rate returns exhibit a sizeable amount of tail dependence. For that reason, student \textit{t}-copulas with low degrees of freedom are often used to model dependence for such data. Furthermore, in many instances cross-sectional correlations in multivariate financial time series have been observed to vary over time \cite{E02, EHV94, LS95}. Consequently, in some recent studies the dependence between financial variables has been modeled via copulas whose parameters vary themselves according to a latent random process \cite{HM08, HR08, P06a}. The unconditional copula is then a mixture over the underlying parametric family according to a certain probability distribution on the parameter.


In this paper we focus on the specific case of tails of correlation mixtures of elliptical copulas, and more specifically of Gaussian and \textit{t}-copulas. This situation arises for instance in certain multivariate stochastic volatility models \cite{YM06} and in the Stochastic Correlation Auto-Regressive (SCAR) model of \cite{HM08}. In the latter model, the latent cross-sectional correlation $\rho_t$ at time $t$ is described by
\begin{align}
\label{SCAR}
  \gamma_t&=\alpha+\beta \gamma_{t-1} + \sigma \varepsilon_t, &
  \rho_t&= \frac{\exp(2\gamma_t)-1}{\exp(2\gamma_t)+1},
\end{align}
where $\varepsilon_t$, $t \in \mathbb{Z}$, are independent standard normal variables, $|\beta|<1$, and the inverse Fisher transform is chosen to keep $\rho_t$ in $(-1,1)$ at all times. This specification is intuitively reasonable, analytically tractable and has been found to provide an excellent fit to financial data when used in conjunction with a Gaussian copula. 
Correlation mixtures over elliptical copulas being not necessarily elliptical anymore, the setting in our paper is not covered by the literature on tail behavior of elliptical distributions \cite{AFG05, AJ07, H05, H08, H09}.


Let $C$ be a bivariate copula and let $(U, V)$ be a random pair with distribution function $C$. Dependence in the tails can be measured by
\[
  \lambda(u) = u^{-1} \, C(u, u) = \Pr(V < u \mid U < u), \qquad 0 < u \le 1,
\]
and its limit $\lambda = \lim_{u \downarrow 0} \lambda(u)$, the \defn{coefficient of tail dependence}. (Elliptical copulas being symmetric, it suffices to consider lower tail dependence.) We call $\lambda(u)$ the coefficient of \defn{penultimate} tail dependence and we will argue that $\lambda(u)$ may be more informative than $\lambda$ when their difference is large. The distinction between tail dependence at asymptotic and subasymptotic levels has already been made several times in the literature \cite{CHT99, FJS05, H00, LT96}. 

\begin{figure}
\begin{center}
\includegraphics[width=1\textwidth]{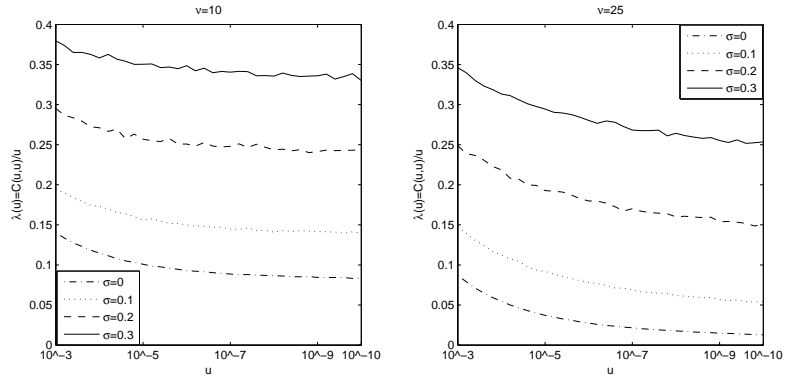}
\caption{\it Coefficient of penultimate tail dependence $\lambda(u) = u^{-1} \, C(u, u)$ for correlation mixtures of \textit{t}-copulas. Mind the logarithmic scale for $u$.}\label{fig1}
\end{center}
\end{figure}

For \textit{t}-copulas with correlations generated by \eqref{SCAR}, $\lambda(u)$ is plotted in Figure \ref{fig1} for values of $u$ that go deep into the joint tail. The unconditional correlation is equal to $\expec[\rho_t] = 0.5$ for all cases. The figure illustrates our two main findings quite well:
\begin{enumerate}
\item Allowing for random correlation greatly increases both $\lambda(u)$ and $\lambda$.
\item The speed of convergence of $\lambda(u)$ to $\lambda$ may be extremely slow, implying that tail dependence at penultimate levels may be significantly stronger than in the limit.
\end{enumerate}
Similar effects can be observed for correlation mixtures of Gaussian copulas, for which $\lambda(u)$ is much larger than its limit $0$ even at levels $u$ that are several orders of magnitude smaller than the ones relevant for practice. Our findings underline the importance of proper modeling of correlation dynamics and of taking threshold sensitivity of tail dependence measures into account. A surprising consequence is that Gaussian copulas with time-varying correlations do a great job in capturing tail dependence at subasymptotic thresholds.

The rest of the paper is structured as follows. In Section~\ref{Sec:mixcop} the impact of random variation in the correlation parameter on tail dependence measures of elliptical copulas is investigated. The (lack of) speed of convergence of $\lambda(u)$ to $\lambda$ for correlation mixtures of Gaussian and \textit{t}-copulas is studied in Section~\ref{Sec:PUTD}. In the Gaussian case, the rate of convergence to zero is so slow that it is actually appropriate to speak of \defn{near asymptotic dependence}. Consequences of our results for modeling are discussed in Section~\ref{Sec:CM} through a Monte Carlo study and a case study. Finally, Section~\ref{Sec:concl} concludes. Proofs are relegated to the Appendix.

\section{The impact of correlation dynamics}\label{Sec:mixcop}

Let $(X, Y)$ be a (standardized) \defn{bivariate elliptical random vector}, that is,
\begin{equation}
\label{E:XY}
  (X, Y) \eqd \bigl(S_1, \rho S_1 + (1 - \rho^2)^{1/2} S_2 \bigr)
\end{equation}
where $-1 \le \rho \le 1$ and where
\begin{equation}
\label{E:spherical}
  (S_1, S_2) \eqd (R \cos \Theta, R \sin \Theta),
\end{equation}
the random variables $R$ and $\Theta$ being independent with $R > 0$ and $\Theta$ uniformly distributed on $(-\pi, \pi)$; see for instance \cite{CHS81, FKN90, KBJ00}. Consider the distribution of the radius $R$ as fixed. Let $C_\rho$ be the copula of $(X, Y)$ seen as parameterized by the correlation parameter $\rho$, that is,
\[
  C_\rho(u, v) = \Pr(U \le u, V \le v), \qquad (u, v) \in [0, 1]^2,
\]
with $U = F(X)$ and $V = F(Y)$, where $F$ is the common marginal distribution function of $X$, $Y$, $S_1$ and $S_2$. Let 
\begin{equation}
\label{E:lambda_rho(u)}
  \lambda_\rho(u) = u^{-1} \, C_\rho(u, u) = \Pr(U < u \mid V < u), \qquad 0 < u \le 1,
\end{equation}
denote the coefficient of penultimate tail dependence and let $\lambda_\rho$ denote its limit as $u \downarrow 0$, provided it exists.

As in the SCAR model \eqref{SCAR}, suppose now that $\rho$ in \eqref{E:XY} is itself a (latent) random variable, independent of $S_1$ and $S_2$. The copula $C$ of $(X, Y)$ is then a $\rho$-mixture of the copulas $C_\rho$,
\begin{equation}
\label{E:Cmixture}
  C(u, v) = \int_{-1}^1 C_\rho(u, v) \, \mu(\diff\rho),
\end{equation}
with $\mu$ equal to the probability distribution of $\rho$ on $[-1,1]$. The penultimate coefficient of tail dependence of $C$ is simply given by
\begin{equation}
\label{E:PenLambda}
  \lambda(u)
  = u^{-1} \, C(u, u)
  = \int_{-1}^1 u^{-1} \, C_\rho(u, u) \, \mu(\diff\rho)
  = \int_{-1}^1 \lambda_\rho(u) \, \mu(\diff\rho),
  \qquad 0 < u \le 1.
\end{equation}

Let the average correlation parameter be denoted by
\[
  \bar{\rho} = \int_{-1}^1 \rho \, \mu(\diff\rho)
\]
Provided second moments exist, $\bar{\rho}$ is equal to the (unconditional) correlation between $X$ and $Y$. If for some $u \in (0, 1]$ the function $\rho \mapsto \lambda_\rho(u)$ is convex, then by Jensen's inequality,
\begin{equation}
\label{E:penlambda:Jensen}
  \lambda(u) \ge \lambda_{\bar\rho}(u)
\end{equation}
whatever the mixing distribution $\mu$. As a consequence, when correlations are themselves driven by a latent random process, tail dependence may be larger than what one may expect. Taking limits as $u \downarrow 0$ shows that this reasoning applies to the coefficient of tail dependence $\lambda = \int_{-1}^1 \lambda_\rho \, \mu(\diff\rho)$ as well.

It remains to investigate the convexity of $\lambda_\rho(u)$ in $\rho$. 

\begin{proposition}\label{P:penlambda:convex}
Let $\lambda_\rho(u)$ be as in \eqref{E:lambda_rho(u)}. For every $u \in (0, 1/2]$, the function $\rho \mapsto \lambda_\rho(u)$ is convex in $\rho \in [0, 1]$.
\end{proposition}

Graphs of $\lambda_{\rho}(u)$ for fixed $u$ (not shown) suggest that the restriction to $\rho \in [0, 1]$ cannot be avoided. Still, under some mild conditions on the distribution of $R$, the convexity is
actually true on $\rho \in [-1, 1]$ for all sufficiently small $u > 0$.

\begin{proposition}\label{P:penlambda:convex2}
Suppose that the distribution function $F_R$ of $R$ in \eqref{E:spherical} has unbounded support and is absolutely continuous with density $f_R$. If
\begin{equation}
\label{E:VonMises}
  \liminf_{r \to \infty} \frac{r \, f_R(r)}{\overline{F}_R(r)} > 1,
\end{equation}
then there exists $u_0 \in (0, 1/2)$ such that the function $\rho \mapsto \lambda_\rho(u)$ is convex in $\rho \in [-1, 1]$ for every $u \in (0, u_0)$. If the limit $\lambda_\rho = \lim_{u \downarrow 0} \lambda_\rho(u)$ exists, then $\lambda_\rho$ is convex in $\rho \in [-1, 1]$.
\end{proposition}

Condition~\eqref{E:VonMises} is verified as soon as the radial density function $f_R$ is regularly varying at infinity of index $-\alpha-1$ for some $\alpha > 1$, in which case the $\liminf$ is actually a limit and is equal to $\alpha$. For the bivariate \textit{t}-distribution with $\nu > 0$ degrees of freedom, the radial density is given by $ f_R(r) = r \, ( 1 + r^2 / \nu)^{-\nu/2-1}$ for $r > 0$. Since $f_R(r) = (c + o(1)) r^{-\nu-1}$ as $r \to \infty$ for some constant $c > 0$, we find $r \, f_R(r) / \overline{F}_R(r) \to \nu$ as $r \to \infty$. Hence, condition~\eqref{E:VonMises} is satisfied if $\nu > 1$. For the bivariate Gaussian distribution, the radial density is given by $f_R(r) = r \, \exp(-r^2/2)$ for $r > 0$, which implies $r \, f_R(r) / \overline{F}_R(r) \to \infty$ as $r \to \infty$, so that \eqref{E:VonMises} is satisfied again.

\begin{figure}
\begin{center}
\includegraphics[width=0.7\textwidth]{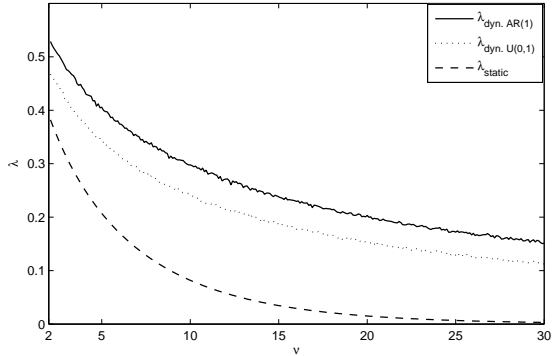}
\caption{\it Coefficients of tail dependence of static and dynamic \textit{t}-copulas.}\label{tcops}
\end{center}
\end{figure}

The difference between $\lambda(u)$ and $\lambda_{\bar\rho}(u)$ can be seen in Figure~\ref{fig1}, whereas Figure~\ref{tcops} illustrates it for the coefficient of tail dependence $\lambda$. For the (static) \textit{t}-copula with $\nu > 0$ degrees of freedom and correlation parameter $\rho \in (-1, 1)$, the coefficient of tail dependence is given in \cite{DMcN05, EMcNS02, S02} by 
\begin{equation}
\label{E:lambda}
  \lambda_{\nu, \rho} = 2 \, t_{\nu+1} \bigl( - \sqrt{\nu+1} \sqrt{1-\rho} / \sqrt{1+\rho} \bigr),
\end{equation}
where $t_\mu$ denotes the cumulative distribution function of the \textit{t}-distribution with $\mu > 0$ degrees of freedom. In Figure~\ref{tcops}, Student \textit{t}-copulas are considered with stochastic and static correlations such that both have an unconditional correlation coefficient of $\bar\rho = 0.5$. The degrees-of-freedom parameter $\nu$ varies from $2$ to $30$. Two different correlation dynamics are considered: the SCAR process in \eqref{SCAR}, and the uniform distribution on $(0, 1)$. The dynamic models clearly lead to significantly higher tail dependence than the static ones. 


\section{Penultimate tail dependence}\label{Sec:PUTD}

So far we have shown that correlation mixtures over elliptical copulas exhibit stronger tail dependence than static ones with the same unconditional correlation. For practical purposes, however, the penultimate coefficient $\lambda(u)$ is more suited to assess the risks of joint extremes than its limit $\lambda$, see for instance \cite{CHT99}. In this section we study the rate of convergence of $\lambda(u)$ to $\lambda$ for correlation mixtures of the Student \textit{t}-copula (Section~\ref{Sec:tcop}) and the Gaussian copula (Section~\ref{Sec:normcop}).

\subsection{Correlation mixtures of \textit{t}-copulas}
\label{Sec:tcop} 

Let $\nu > 0$ and $-1 < \rho < 1$. Let $Z_1, Z_2, S$ be independent random variables such that the
$Z_i$ are standard normal and $\nu S^2$ has a chi-square distribution with $\nu$
degrees of freedom. Put
\begin{align*}
  X &= S^{-1} \, Z_1, & Y = S^{-1} \, \bigl( \rho Z_1 + {\textstyle\sqrt{1 - \rho^2}} Z_2 \bigr).
\end{align*}
The distribution of the random vector $(X, Y)$ is bivariate
\textit{t} with $\nu$ degrees of freedom and correlation parameter
$\rho$. Its copula, $C$, is the bivariate \textit{t}-copula with
parameters $\nu$ and $\rho$.

Now suppose as in Section~\ref{Sec:mixcop} that $\rho$ is itself a random variable with range in $(-1, 1)$ and independent of $Z_1, Z_2, S$. The unconditional
copula, $C$, of $(X, Y)$ is then a correlation mixture of bivariate \textit{t}-copulas with fixed degrees-of-freedom parameter $\nu$.

For nonrandom $\rho \in (-1, 1)$, the coefficient of tail dependence $\lambda_{\nu, \rho}$ is given in \eqref{E:lambda}. For general, random
$\rho$, we have $\lambda = \expec[\lambda_{\nu, \rho}]$, the expectation being with respect to $\rho$.

Figure~\ref{fig1} suggests that the rate of convergence of $\lambda(u) = u^{-1} \, C(u, u)$ to $\lambda$ may be slow, especially when $\nu$ is large. This is confirmed in the following result. For real $x$, let $x_+ = \max(x, 0)$ be its positive part.

\begin{proposition}\label{P:PUT}
Let $C$ be a correlation mixture of bivariate \textit{t}-copulas with degrees-of-freedom parameter $\nu > 0$. We have
\[
  \lambda(u) = u^{-1} \, C(u, u) = \lambda + \gamma \, u^{2/\nu} + o(u^{2/\nu}), \qquad u \downarrow 0,
\]
where $\gamma$ is a positive constant depending on $\nu$ and the distribution of $\rho$ given in \eqref{E:bivtmix:gamma} below.
\end{proposition}

The rate of convergence of $\lambda(u)$ to its limit $\lambda$ is of the order $O(u^{2/\nu})$. The larger $\nu$, the slower this rate. Since $\gamma$ is positive, $\lambda(u)$ may therefore be (much) larger than $\lambda$. As a result, at finite thresholds, the tail may look much heavier than it is in the limit. This in turn may cause estimators of $\nu$ to be negatively biased.

\subsection{Correlation mixtures of Gaussian copulas}
\label{Sec:normcop}

Let $X, Z, \rho$ be independent random variables, $X$ and $Z$ being standard normal and $\rho$ taking values in $(-1, 1)$. Put
\begin{equation}
\label{E:Gaussian:repr}
  Y = \rho X + (1 - \rho^2)^{1/2} Z.
\end{equation}
Conditionally on $\rho$, the distribution of $(X, Y)$ is bivariate normal with zero means, unit variances, and correlation $\rho$. The unconditional correlation of $X$ and $Y$ is $\corr(X, Y) = \expec [\rho] = \bar{\rho}$, and the copula of $(X, Y)$ is a correlation mixture of Gaussian copulas:
\begin{equation}
\label{E:CR}
  C(u, v) = \expec [C_\rho(u, v)], \qquad (u, v) \in [0, 1]^2,
\end{equation}
the expectation being with respect to $\rho$, and with $C_\rho$ denoting the bivariate Gaussian copula with correlation $\rho$. As before, put $\lambda(u) = u^{-1} \, C(u, u)$ and $\lambda_\rho(u) = u^{-1} \, C_\rho(u, u)$.

By the assumption that $\rho < 1$ almost surely and since the coefficient of tail dependence of $C_\rho$ is equal to $\lambda_\rho = \lim_{u \downarrow 0} \lambda_\rho(u) = 0$, the coefficient of tail dependence of $C$ is $\lambda = \expec[\lambda_\rho] = 0$ too. So just like Gaussian copulas, correlation mixtures of Gaussian copulas have asymptotically independent tails.

To measure the degree of tail association in case of asymptotic independence, Ledford and Tawn \cite{LT96} introduced the coefficient
\begin{equation}
\label{E:eta}
  \eta = \lim_{u \downarrow 0} \frac{\log u}{\log C(u, u)} \in [0, 1],
\end{equation}
the existence of the limit being an assumption. In this case, one can write
\begin{equation}
\label{E:LT}
  C(u, u) = u^{1/\eta} \, L(u),
\end{equation}
the function $L$ satisfying $\lim_{u \downarrow 0} \log L(u) / \log u = 0$, that is, for all $\eps > 0$ we have
\begin{align*}
  \lim_{u \downarrow 0} u^\eps \, L(u) &= 0, & \lim_{u \downarrow 0} u^{-\eps} \, L(u) &= \infty.
\end{align*}
Typically, the function $L$ is slowly varying at zero: $\lim_{u \downarrow 0} L(ux) / L(u) = 1$ for all $x > 0$. The coefficient $\eta$ is related to the coefficient $\bar{\chi}$ in \cite{CHT99} through $\bar{\chi} = 2 \eta - 1$.

The pair of coefficients $(\lambda, \eta)$ measures the amount of tail dependence across the classes of asymptotic dependence and independence. The following two cases are most common:
\begin{enumerate}
\item Asymptotic dependence: If $\lambda > 0$, then necessarily $\eta = 1$.
\item Asymptotic independence: If $\eta < 1$, then $\lambda = 0$.
\end{enumerate}

For the bivariate Gaussian copula with correlation $\rho < 1$, for instance, we have $\lambda = 0$ and $\eta = (\rho + 1)/2$, so $\bar{\chi} = \rho$. One may then be tempted to believe that for correlation mixtures of Gaussian copulas the coefficient $\bar{\chi}$ is equal to $\bar{\rho}$. However, this is false. Instead, a new situation is encountered, one that is in between the two cases described in the preceding paragraph.

\begin{proposition}
\label{P:CR} Let $C$ be as in \eqref{E:CR}. If $\rho < 1$ almost surely but the upper endpoint of the distribution of $\rho$ is equal to $1$, then for every $\eps > 0$,
\[
  \lim_{u \downarrow 0} u^{-\eps} \, \lambda(u) = \infty.
\]
As a consequence, although we have $\lambda = 0$ (asymptotic independence), the Ledford--Tawn index is equal to $\eta = 1$; likewise, $\bar{\chi} = 1$.
\end{proposition}

According to Proposition~\ref{P:CR}, the rate of decay of $\lambda(u)$ to $\lambda = 0$ is slower than any positive power of $u$. So even though the tails of $C$ are asymptotically independent, we are as close as one can get to the case of tail dependence. Therefore, we coin this type of tail behavior \emph{near asymptotic dependence}. As far as we know, this situation has not yet been encountered in the literature; see for instance the extensive list of examples in \cite{H00}. Note that the function $L(u)$ in \eqref{E:LT} is equal to $\lambda(u)$.

\begin{figure}
\begin{center}
\includegraphics[width=1\textwidth]{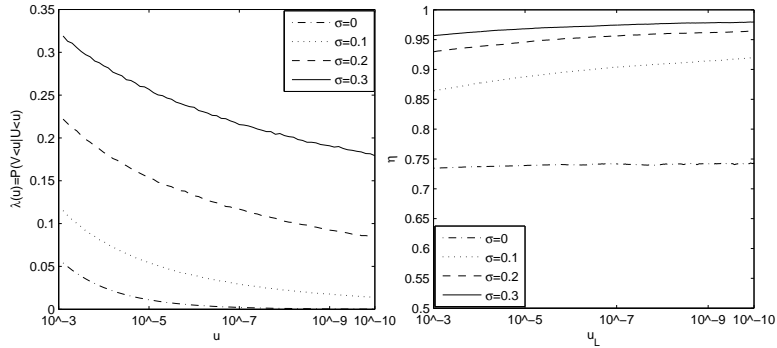}
\end{center}
\caption{\it The coefficients $\lambda(u)$ and $\eta$ for correlation mixtures of the Gaussian copula.}\label{fig2}
\end{figure}

By way of example, the left-hand panel in Figure~\ref{fig2} shows $\lambda(u)$ for tail probabilities $u$ ranging from $10^{-3}$ to $10^{-10}$,  the correlation parameter being driven by the SCAR model \eqref{SCAR} for various choices of the parameters, but always with $\bar{\rho} = 0.5$. Especially when the variation in $\rho$ is large, the coefficient $\lambda(u)$ remains quite sizeable even very far in the tail. To investigate the speed of decline of $\lambda(u)$ numerically, we write
\[
  \log \lambda(u) = \beta \, \log u + \log L(u)
\]
with $\beta = 1/\eta - 1$ and we treat $\log L(u)$ as a constant, its rate of change being much slower than the one of $\log u$. The slope $\beta$ can be estimated by a simple least squares regression, from which $\eta$ can be recovered. We estimated $\eta$ in this way by making $u$ vary over intervals $[u_L, u_U]=[u^{-k}, u^{-k-3}]$ for $k$ ranging from $3$ to $10$ with steps of size $0.01$. The corresponding estimates of $\eta$ as a function of $u_L$ are depicted in the right-hand panel Figure \ref{fig2}. For the constant correlation case we find indeed $\eta \approx (1 + \rho)/2 = 0.75$. In contrast, when correlation is not constant, $\eta$ seems to be converging to $1$ as $u \downarrow 0$.

\section{Consequences for modeling}
\label{Sec:CM}

In this section we illustrate the consequences of our results for modeling data with static versus dynamic \textit{t}-copulas. Of particular interest is the effect on inference on tail dependence when correlation is falsely assumed to be constant. To fit a static \textit{t}-copula, we estimate $\nu$ by maximum likelihood and $\rho$ by the method of moments, the likelihood function for $\nu$ being given by
\begin{align}
L(\nu \mid u, v) = \boldsymbol{t}_{\nu}\bigl(t_{\nu}^{-1}(u), \, t_{\nu}^{-1}(v), \hat{\rho} \bigr),
\end{align}
where $\boldsymbol{t}$ denotes the density of the bivariate \textit{t}-distribution and $\hat{\rho}$ is the empirical correlation of the pairs $\bigl(t_{\nu}^{-1}(U_i), \, t_{\nu}^{-1}(V_i) \bigr)$.

\subsection{Simulation study}

We simulated data from a \textit{t}-copula with a fixed degrees-of-freedom parameter $\nu$ but stochastic correlations. The correlations were drawn from the SCAR process in \eqref{SCAR}. We chose $\beta=0.97$ and $\sigma=0.05, 0.1, 0.15, 0.2$, whereas $\alpha$ was chosen such that the unconditional correlation was in all cases equal to $\bar{\rho} = 0.5$. The values for the degrees-of-freedom parameter $\nu$ were $5$, $10$, $50$, and $\infty$, the last case yielding the Gaussian copula. The sample size and the number of Monte Carlo replications were chosen to be equal to $1000$.  For each of the $1000$ simulated data sets we fitted a static \textit{t}-copula and computed both the implied coefficient of tail dependence $\lambda_{\hat{\nu}, \hat{\rho}}$ as well as the one of penultimate tail dependence $\lambda_{\hat{\nu},\hat{\rho}}(u)$ at $u = 0.01$, a tail probability of potential practical interest.

\begin{table}[t]
  \centering
  \caption{Bias estimating $\lambda(u)$ and $\lambda$ using \textit{t}-copulas}\label{bias_mc}
    \begin{tabular}{lcccc}
\hline\hline
    $\lambda(u)$  & $\sigma=0.05$ & $\sigma=0.1$ & $\sigma=0.15$ & $\sigma=0.2$ \\
    $\nu=\phantom{0}5$  & -0.011 & -0.001 & 0.002 & -0.022 \\
    $\nu=10$ & -0.027 & 0.014 & 0.022 & 0.008 \\
    $\nu=20$ & -0.082 & 0.005 & 0.029 & 0.031 \\
    $\nu=\infty$ & -0.096 & 0.000 & 0.034 & 0.030 \\
    \hline
    $\lambda$     &       &       &       &  \\
    $\nu=\phantom{0}5$  & 0.002 & 0.029 & 0.029 & 0.012 \\
    $\nu=10$ & 0.027 & 0.081 & 0.107 & 0.097 \\
    $\nu=20$ & 0.039 & 0.142 & 0.213 & 0.235 \\
    $\nu=\infty$ & 0.022 & 0.131 & 0.240 & 0.303 \\
\hline\hline\multicolumn{5}{p{9cm}}{\footnotesize \textbf{Note.} This table reports on the bias for $\lambda$ and $\lambda(0.01)$ for \textit{t}-copulas
when the correlation is assumed to be constant when they are in fact generated by the SCAR model~\eqref{SCAR}. The sample size and the number of Monte Carlo
replications are equal to 1000.}
 \end{tabular}
\end{table}

The bias for the estimates of $\lambda$ and $\lambda(u)$ is reported in Table~\ref{bias_mc}. Overall, $\lambda(u)$ is estimated much more accurately than $\lambda$ and both are estimated best when $\nu$ is low. The bias generally increases in $\nu$ and for $\lambda$ it also increases in the volatility $\sigma$ of the correlation process. 
Notably, the estimates suggest quite strong tail dependence even when the true conditional copula is Gaussian. The estimates of $\rho$ and $\nu$ (not shown) show that the positive bias in $\lambda$ is due a negative bias in $\nu$, whereas the overall correlation $\bar{\rho} = 0.5$ is only slightly underestimated with average estimates ranging from $0.48$ to $0.5$. It seems as though the estimate of the degrees-of-freedom parameter $\nu$ is such that the implied coefficient of penultimate tail dependence $\lambda(0.01)$ matches the true one and that the consequence of this is a severe underestimation of the coefficient of tail dependence $\lambda$.

These findings continued to hold for other sample sizes and other correlation-driving processes.

\subsection{Application to stock market and exchange rate returns}

We consider international stock market returns at daily and monthly frequencies and daily exchange rate returns. The data sets are daily returns of the Dow Jones industrial index (DJ) and the NASDAQ composite index (NQ) from March 26, 1990 until March 23, 2000, daily returns of the MSCI index for France (Fra) and Germany (Ger) from October 6, 1989 until October 17, 2008, monthly returns of the Datastream stock index for Germany, Japan (Jap), the UK and the US from January 1974 until May 2008, and daily exchange rate returns of the Euro (EUR), British pound (GBP) and Japanes Yen (JPY) against the US dollar from January 1, 2005 until December 31, 2008. Returns are calculated as 100 times the first difference of the natural logarithm of prices.

To model the marginal distributions, we opt for the stochastic volatility model \cite{C73, T86} because of its natural connection to the SCAR model for the correlation dynamics \eqref{SCAR}. The basic stochastic volatility model for the return $r_{t}$ at time $t=1,\ldots,T$ is given by
\begin{align*}
  r_{t} &= \exp(h_{t}/2) \, \epsilon_{t}, &
  h_{t} &= \alpha_h + \beta_h \, h_{t-1} + \tau \, \eta_{t},
  \end{align*}
where $\epsilon_{t}$ and $\eta_{t}$ are independent standard normal random variables, uncorrelated with the innovations driving the dependence process \eqref{SCAR}. Estimation of the model is done by simulated maximum likelihood using the efficient importance sampler \cite{LR03}. The static \textit{t}-copulas are fitted as described above, whereas for the time-varying model we condition the estimate of $\nu$ on the correlations $\hat{\rho}_t$ estimated using the SCAR model, yielding estimates $\hat{\rho}_t$ and $\hat{\nu}|\hat{\rho}_t$ respectively.

We also compute the implied coefficient of tail dependence $\lambda$ and the coefficient of penultimate tail dependence $\lambda(u)$ at various economically interesting levels. In particular, the levels we consider correspond to exceedances that are expected to occur once a year, once a decade and once a century. We denote the corresponding estimates by $\lambda_{\text{year}}$, $\lambda_{\text{dec}}$ and $\lambda_{\text{cent}}$. These can be interpreted as the probabilities that one market makes a certain large downward movement conditionally on the other market doing the same. The coefficient of tail dependence $\lambda$, on the other hand, denotes the probability of one market crashing completely, i.e.\ dropping to a level of zero, conditional on the same event for the other market. The latter scenario is economically rather unrealistic and is, in our view, relevant only for individual stocks, not for complete markets.

The results can be found at the end in Table~\ref{T:res}. The most striking finding is that the estimated degrees-of-freedom parameter $\nu$ is significantly larger when correlations are allowed to be stochastic, and in many cases even virtually infinity, corresponding to the Gaussian copula (we report $\infty$ whenever the upper bound of $400$ in the optimization routine was obtained). This implies that a large part of the fat-tailedness can be captured by random correlations. For the estimated measures of tail dependence, two things are notable. First, the penultimate version $\lambda(u)$ is much larger than the limiting one $\lambda$. This suggests that $\lambda$ may be a too optimistic measure for assessing the risk of spillovers of large downward movements across financial markets. Second, although the fitted models that allow for 
time-varying correlations have a lower limiting coefficient of tail dependence than the static ones, at practically relevant quantiles these models capture the dependence in the tails of the distribution quite well. So just like in the simulation above, the static and the dynamic models both do a good job in matching penultimate tail dependence, but the static \textit{t}-copulas may lead to overestimation of the coefficient of tail dependence $\lambda$.

\section{Conclusions}\label{Sec:concl}

We have studied tail dependence properties of correlation mixtures of elliptical copulas, a situation which occurs when the correlation parameter is itself driven by a latent random process. We have shown that the coefficient of (penultimate) tail dependence is larger than for ordinary elliptical copulas with the same unconditional correlation. Furthermore, for Gaussian copulas and \textit{t}-copulas, tail dependence at sub-asymptotic levels quantiles can be substantially larger than in the limit. In a simulation study we found that ignoring the dynamic nature of correlations when estimating \textit{t}-copulas leads to biased estimates of the coefficient of tail dependence. Our empirical application showed that estimates of the degrees-of-freedom parameter of a \textit{t}-copula are much lower when assuming a static correlation than when conditioning on dynamic correlations. At the same time the models based on dynamic correlations produce similar dependence in the tails at economically relevant quantiles, but lower tail dependence in the limit.

A notable discovery was that under some fairly weak conditions, correlation mixtures of Gaussian copulas have tails that are so close to being asymptotically dependent that they give rise to the newly defined category of \defn{near asymptotic dependence}. In practice this implies that it is virtually impossible to distinguish such copulas from ones that have asymptotically dependent tails such as \textit{t}-copulas. 

These findings suggest that for practical purposes the Gaussian copula is more attractive than as often stated in the literature, as long as one accounts for the (empirically observed) fact that correlations vary over time. This finding can be seen as an analogue to the effect that conditionally Gaussian models for time-varying
volatility such as GARCH and stochastic volatility models can create heavy tails in the margins. Thus, conditionally Gaussian models are more than just a simplifying approximation in a multivariate setting, as they are able to capture the tails both in the margins as in the copula.

\section*{Acknowledgments}

We would like to thank the participants of the ``Workshop on Copula Theory and Its Applications'' (Warsaw, 2009) for comments. The first author gratefully acknowledges a PhD traveling grant from METEOR. The second authors' research was supported by IAP research network grant nr.\ P6/03 of the Belgian government (Belgian Science Policy) and by contract nr.\ 07/12/002 of the Projet d'Actions de Recherche Concert\'ees of the Communaut\'e fran\c{c}aise de Belgique, granted by the Acad\'emie universitaire Louvain.

\begin{appendix}
\section{Proofs}
\begin{proof}[Proof of Proposition~\ref{P:penlambda:convex}]
Write $\rho = \sin \gamma \in [-1, 1]$ for $\gamma = \arcsin \rho
\in [-\pit, \pit]$. Then $(1 - \rho^2)^{1/2} = \cos \gamma \in [0,
1]$, so that
\[
  Y 
  = \rho S_1 + (1 - \rho^2)^{1/2} S_2
  = R \, \sin \gamma \cos \Theta + R \, \cos \gamma \sin \Theta
  = R \, \sin(\Theta + \gamma),
\]
yielding the representation
\[
  (X, Y) = \bigl(R \, \cos \Theta, \, R \, \sin(\Theta + \gamma) \bigr).
\]
Let $t \ge 0$ be such that $1 - F(t) = F(-t) = u$, where $F$ is the
marginal distribution function of $X$ and $Y$. Since the
distribution of $(X, Y)$ is symmetric around zero and since $R$ and
$\Theta$ are independent,
\[
  C_\rho(u, u)
  = \Pr(X > t, Y > t) 
  = \int_t^\infty \Pr\biggl[ \cos \Theta > \frac{t}{r}, \sin(\Theta + \gamma) > \frac{t}{r} \biggr] \, \diff F_R(r),
\]
with $F_R$ the distribution function of $R$. As a consequence, it is
sufficient to show that for fixed $z \in [0, 1)$, the function
\[
  \rho \mapsto \Pr[ \cos \Theta > z, \sin(\Theta + \gamma) > z]
\]
is convex in $\rho \in [0, 1]$, where $\gamma = \arcsin \rho \in [0,
\pit]$ and with $\Theta$ uniformly distributed on $(-\pi, \pi)$.
Write $z = \cos \alpha \in [0, 1)$, so $\alpha = \arccos z \in (0,
\pi]$. Then
\[
  \cos \Theta > \cos \alpha \Longleftrightarrow  - \alpha < \Theta < \alpha
\]
whereas
\[
  \sin (\Theta + \gamma) > \cos \alpha = \sin(\pit - \alpha) \Longleftrightarrow \pit - \alpha < \Theta + \gamma < \pit + \alpha.
\]
Joining these two double inequalities and using the fact that $\pit
- \gamma \ge 0$ yields
\[
  \Pr[ \cos \Theta > z, \sin(\Theta + \gamma) > z ]
  = \Pr(- \alpha + \pit - \gamma < \Theta < \alpha)
  = \frac{1}{2\pi} \max(2\alpha - \pit + \gamma, 0).
\]
Since $\gamma = \arcsin \rho$ is convex in $\rho \in [0, 1]$, the
result follows.
\end{proof}

\begin{proof}[Proof of Proposition~\ref{P:penlambda:convex2}]
We keep the same notations as in the proof of
Proposition~\ref{P:penlambda:convex}. Now let $0 < u < 1/2$ so that
$t = F^{-1}(1-u) > 0$. We have
\[
  C_\rho(u, u)
  = \int_t^\infty \frac{1}{2\pi} \max(2\alpha - \pit + \gamma, 0) \, \diff F_R(r)
\]
where $\alpha = \arccos(t/r)$ and $\gamma = \gamma(\rho) = \arcsin
\rho$. Observe that
\[
  2\alpha - \pit + \gamma > 0 \Longleftrightarrow r > \frac{t}{\cos \bigl( (\pit - \gamma)/2 \bigr)} = r(\rho).
\]
We find
\begin{align*}
  2\pi \, \frac{\diff}{\diff\rho} C_\rho(u, u) &= \overline{F}_R \bigl( r(\rho) \bigr) \, \frac{\diff \gamma(\rho)}{\diff\rho}, \\
  2\pi \, \frac{\diff^2}{\diff\rho^2} C_\rho(u, u)
  &= - f_R\bigl( r(\rho) \bigr) \, \frac{\diff r(\rho)}{\diff \rho} \, \frac{\diff \gamma(\rho)}{\diff \rho}
  + \overline{F}_R \bigl( r(\rho) \bigr) \, \frac{\diff^2 \gamma(\rho)}{\diff\rho^2}.
\end{align*}
We have
\begin{align*}
  \frac{\diff \gamma(\rho)}{\diff\rho} &= (1 - \rho^2)^{-1/2}, &
  \frac{\diff^2 \gamma(\rho)}{\diff\rho^2} &= \frac{\rho}{1 - \rho^2} \, \frac{\diff \gamma(\rho)}{\diff\rho}.
\end{align*}
Some goniometric juggling yields $\cos \{ (\pit - \gamma) / 2 \} = \{ (1 + \rho) / 2 \}^{1/2}$, whence
\begin{align*}
  r(\rho) &= 2^{1/2} t \, (1+\rho)^{-1/2}, &
  \frac{\diff r(\rho)}{\diff \rho} &= - \frac{1}{2} \, \frac{1}{1+\rho} \, r(\rho).
\end{align*}
Writing $\nu(r) = r \, f_R(r) / \overline{F}(r)$, we conclude that
\begin{align*}
  2\pi \, \frac{\diff^2}{\diff\rho^2} C_\rho(u, u) = \overline{F} \bigl( r(\rho) \bigr) \, \frac{\diff \gamma(\rho)}{\diff \rho} \, \frac{1}{1+\rho}
  \biggl( \frac{1}{2} \, \nu \bigl( r(\rho) \bigr) + \frac{\rho}{1 - \rho} \biggr)
\end{align*}
For $\rho < 0$, we have $r(\rho) > t$. If $u$ is small enough so
that $t$ is large enough so that $\nu(r) > 1$ for all $r \ge t$,
then the factor between big brackets on the right-hand side of the
last display is positive for all $\rho \in (-1, 1)$. Hence
$C_\rho(u, u)$ is convex in $\rho$.
\end{proof}
\begin{proof}[Proof of Proposition~\ref{P:PUT}]
Let $x > 0$. We have
\begin{align*}
  \Pr(X > x, \, Y > x)
  &= \Pr \bigl( S^{-1} Z_1 > x, \, S^{-1} ( \rho Z_1 + {\textstyle\sqrt{1 - \rho^2}} Z_2 ) > x \bigr) \\
  &= \Pr \bigl( \min(Z_1, \rho Z_1 + \textstyle{\sqrt{1 - \rho^2}} Z_2) S^{-1} > x \bigr).
\end{align*}
Put $W = \min(Z_1, \rho Z_1 + \sqrt{1 - \rho^2} Z_2)$ and $W_+ =
\max(W, 0)$. Then we can rewrite the above equation as $\Pr(X > x, \, Y > x) = \Pr( W_+ S^{-1} > x )$. We need precise information on the upper tail of $S^{-1}$: see
Lemma~\ref{L:S}. We can then proceed as follows:
\begin{align*}
  \lefteqn{
  \Pr(X > x, \, Y > x) = \int_{(0, \infty)} \Pr(S^{-1} > x/w) \, d\Pr(W \le w)
  } \\
  &= \int_{(0, \infty)} a_\nu \, (x/w)^{-\nu} \, \bigl( 1 - b_\nu \, (x/w)^{-2} \bigr) \, d\Pr(W \le w) \\
  &\qquad + \int_{(0, \infty)} \Delta(x/w) \, d\Pr(W \le w) \\
  &= a_\nu \, x^{-\nu} \, \bigl( \E[W_+^\nu] - b_\nu \, x^{-2} \, \E[W_+^{\nu+2}] \bigr)
  + \int_{(0, \infty)} \Delta(x/w) \, d\Pr(W \le w).
\end{align*}
The remainder term is
\[
  0 \le \int_{(0, \infty)} \Delta(x/w) \, d\Pr(W \le w) \le c_\gamma \, x^{-\nu-4} \, \E[W_+^{\nu+4}].
\]
We find that, as $x \to \infty$,
\[
  \Pr(X > x, \, Y > x)
  = a_\nu \, x^{-\nu} \, \bigl( \E[W_+^\nu] - b_\nu \, x^{-2} \, \E[W_+^{\nu+2}] + O(x^{-4}) \bigr).
\]

The marginal tail of $X$ can be represented in the same way: it
suffices to replace $W$ in the preceding display by a standard
normal random variable $Z$, that is, as $x \to \infty$,
\[
  \Pr(X > x) = \Pr(S^{-1} Z > x) = a_\nu \, x^{-\nu} \, \bigl( \E[Z_+^\nu] - b_\nu \, x^{-2} \, \E[Z_+^{\nu+2}] + O(x^{-4}) \bigr).
\]
It follows that as $x \to \infty$,
\begin{align}
\label{E:bivtmix:tail}
  \lefteqn{
  \Pr(Y > x \mid X > x) = \frac{\Pr(X > x, \, Y > x)}{\Pr(X > x)}
  } \nonumber \\
  &= \frac{\E[W_+^\nu] - b_\nu \, x^{-2} \, \E[W_+^{\nu+2}] + O(x^{-4})}{\E[Z_+^\nu] - b_\nu \, x^{-2} \, \E[Z_+^{\nu+2}] + O(x^{-4})} \nonumber \\
  &= \frac{\E[W_+^\nu]}{\E[Z_+^\nu]} + b_\nu \, x^{-2} \, \frac{\E[W_+^\nu] \E[Z_+^{\nu+2}] - \E[W_+^{\nu+2}] \E[Z_+^\nu]}{\bigl(\E[Z_+^\nu]\bigr)^2} + O(x^{-4}).
\end{align}
As a consequence, the coefficient of tail dependence is given by
\begin{equation}
\label{E:bivtmix:lambda}
  \expec[\lambda_{\nu, \rho}] = \lambda = \lim_{x \to \infty} \Pr(Y > x \mid X > x) = \frac{\E[W_+^\nu]}{\E[Z_+^\nu]}.
\end{equation}

Moreover, the expansion in \eqref{E:bivtmix:tail} gives us a handle
on the rate of convergence of $\lambda(u)$ towards $\lambda$. Let
$t_\nu^{-1}$ be the quantile function of the \textit{t}-distribution
with $\nu$ degrees of freedom. By symmetry of the upper and lower
tails,
\begin{multline*}
  \lambda(u)
  = \Pr \bigl( Y > t_\nu^{-1}(1-u) \mid X > t_\nu^{-1}(1-u) \bigr) \\
  = \lambda + b_\nu \, \frac{\E[W_+^\nu] \E[Z_+^{\nu+2}] - \E[W_+^{\nu+2}] \E[Z_+^\nu]}{\bigl(\E[Z_+^\nu]\bigr)^2} \,
  \bigl(t_\nu^{-1}(1-u)\bigr)^{-2} \\
  + O \Bigl( \bigl(t_\nu^{-1}(1-u)\bigr)^{-4} \Bigr)
\end{multline*}
Since $1 - t_\nu(x) = \Pr(X > x) \sim a_\nu \, x^{-\nu} \,
\E[Z_+^\nu]$ as $x \to \infty$, we have
\[
  t_\nu^{-1}(1 - u) \sim \bigl( a_\nu \, \E[Z_+^\nu] \bigr)^{1/\nu} u^{-1/\nu}, \qquad u \downarrow 0.
\]
(By $f(y) \sim g(y)$ we mean that $f(y) / g(y) \to 1$.) We obtain $\lambda(u) = \lambda + \gamma \, u^{2/\nu} + o(u^{2/\nu})$ with
\[
  \gamma = b_\nu \, \frac{\E[W_+^\nu] \E[Z_+^{\nu+2}] - \E[W_+^{\nu+2}] \E[Z_+^\nu]}{\bigl(\E[Z_+^\nu]\bigr)^2} \, \bigl( a_\nu \, \E[Z_+^\nu] \bigr)^{-2/\nu}.
\]
Equating \eqref{E:bivtmix:lambda} and \eqref{E:lambda} (for
non-random $\rho$) we find
\[
  \expec[W_+^\nu] = \expec[Z_+^\nu] \, \expec[\lambda_{\nu, \rho}],
\]
the latter expectation being with respect to the random variable
$\rho$. As a consequence,
\[
  \frac{\E[W_+^\nu] \E[Z_+^{\nu+2}] - \E[W_+^{\nu+2}] \E[Z_+^\nu]}{\bigl(\E[Z_+^\nu]\bigr)^2}
  = \frac{\E[Z_+^{\nu+2}]}{\E[Z_+^\nu]} \expec[\lambda_{\nu, \rho} - \lambda_{\nu+2, \rho}],
\]
the sign of which is positive, for $\lambda_{\nu, \rho}$ is decreasing in $\nu$. Furthermore, by~\eqref{E:Znu} below and the identity
$\Gamma(z+1) = z \, \Gamma(z)$,
\[
  \frac{\E[Z_+^{\nu+2}]}{\E[Z_+^\nu]}
  = \frac{\frac{2^{\nu/2}}{\sqrt{\pi}} \, \Gamma \bigl( (\nu+3)/2 \bigr)}{\frac{2^{\nu/2-1}}{\sqrt{\pi}} \, \Gamma \bigl( (\nu+1)/2 \bigr)}
  = 2 \, \frac{\nu+1}{2} = \nu + 1.
\]
We obtain
\begin{equation}
\label{E:bivtmix:gamma}
  \gamma = \bigl( a_\nu \, \E[Z_+^\nu] \bigr)^{-2/\nu} \, b_\nu \, (\nu + 1) \, \E[\lambda_{\nu, \rho} - \lambda_{\nu+2, \rho}].
\end{equation}
with $a_\nu$ and $b_\nu$ given in \eqref{E:anubnu} and 
\begin{equation}
\label{E:Znu}
  \E[Z_+^\nu] = \frac{1}{\sqrt{2\pi}} \int_0^\infty z^\nu \, e^{-z^2/2} \, dz = \frac{2^{\nu/2-1}}{\sqrt{\pi}} \, \Gamma \bigl( (\nu+1)/2 \bigr).
\end{equation}
\end{proof}

\begin{lemma}
\label{L:S} For $\nu > 0$, let $S$ be a positive random variable
such that $\nu S^2$ has a chi-square distribution with degrees of
freedom $\nu$. Then for all $z > 0$,
\[
  \Pr(S^{-1} > z) = a_\nu \, z^{-\nu} \bigl(1 - b_\nu \, z^{-2}\bigr) + \Delta(z)
\]
where
\begin{align}
\label{E:anubnu}
  a_\nu &= \frac{(\nu/2)^{\nu/2}}{\Gamma(\nu/2+1)}, & \qquad b_\nu &= \frac{(\nu/2)^2}{\nu/2+1},
\end{align}
and $0 \le \Delta(z) \le c_\gamma \, z^{-\nu-4}$ for some positive
constant $c_\nu$.
\end{lemma}

\begin{proof}
Since $\nu S^2$ is chi-squared with $\nu$ degrees of freedom,
\[
  \Pr(S^{-1} > z)
  = \Pr( \nu S^2 < \nu z^{-2})
  = \int_0^{\nu z^{-2}} \frac{1}{\Gamma(\nu/2) \, 2^{\nu/2}} y^{\nu/2 - 1} \, e^{-y/2} \, dy.
\]
For $z \ge 0$, we have $e^{-z} - (1 - z) = \int_0^z (1 - e^{-y}) \,
dy = \int_0^z \int_0^y e^{-x} \, dx \, dy$ and thus $0 \le e^{-z} -
(1-z) \le z^2/2$. We find that
\begin{align*}
  0
  &\le \Pr(S^{-1} > z) - \frac{1}{\Gamma(\nu/2) \, 2^{\nu/2}}  \int_0^{\nu z^{-2}} y^{\nu/2 - 1} \, (1 - y/2) \, dy \\
  &\le \frac{1}{\Gamma(\nu/2) \, 2^{\nu/2}}  \int_0^{\nu z^{-2}} y^{\nu/2 - 1} \, \frac{(y/2)^2}{2} \, dy
  = \text{constant} \, z^{-\nu - 4}
\end{align*}
the constant depending on $\nu$. The integral can be computed as
follows:
\begin{multline*}
  \frac{1}{\Gamma(\nu/2) \, 2^{\nu/2}}  \int_0^{\nu z^{-2}} y^{\nu/2 - 1} \, (1 - y/2) \, dy \\
  = \frac{1}{\Gamma(\nu/2) \, 2^{\nu/2}}  \biggl( \frac{(\nu z^{-2})^{\nu/2}}{\nu/2} - \frac{1}{2} \frac{(\nu z^{-2})^{\nu/2+1}}{\nu/2+1} \biggr) 
  = a_\nu \, z^{-\nu} \, \bigl( 1 - b_\nu z^{-2} \bigr)
\end{multline*}
with $a_\nu$ and $b_\nu$ as in \eqref{E:anubnu}.
\end{proof}

\begin{proof}[Proof of Proposition~\ref{P:CR}]
Fix $\eps > 0$. We have to prove that $u^{1+\eps} = o\bigl( C(u, u)\bigr)$ as $u \downarrow 0$. Recall the representation $(X, Y) = \bigl( X, \rho X + (1 - \rho^2)^{1/2} Z \bigr)$ in \eqref{E:Gaussian:repr}. Fix $0 < \rho_0 < 1$. By symmetry, if $0 < u < 1/2$,
\begin{align*}
    C(u, u)
    &= \Pr[X > \Phi^{-1}(1-u), \, \rho X + \sqrt{1-\rho^2} Z > \Phi^{-1}(1-u)] \\
    &\ge \Pr[X > \Phi^{-1}(1-u), \, \rho_0 X > \Phi^{-1}(1-u), \, \rho > \rho_0, \, Z > 0] \\
    &= \frac{1}{2} \, \Pr(\rho > \rho_0) \, \Pr[X > \Phi^{-1}(1-u)/\rho_0].
\end{align*}
Let $\varphi(x) = (2\pi)^{-1/2} \exp(-x^2/2)$ denote the standard normal density function and write $\bPhi = 1-\Phi$. From Mill's ratio, $\bPhi(x) = \bigl(1 + o(1)\bigr) \, x^{-1} \, \varphi(x)$ as $x \to \infty$, it follows that $\Phi^{-1}(1-u) = \bigl(1 + o(1)\bigr) \, (-2\log u)^{1/2}$ as $u \downarrow 0$ and therefore
\begin{align*}
    \Pr[X > \Phi^{-1}(1-u)/\rho_0]
    &= \bPhi \bigl( \Phi^{-1}(1-u)/\rho_0 \bigr) \\
    &= \bigl(1 + o(1)\bigr) \frac{\varphi \bigl( \Phi^{-1}(1-u)/\rho_0 \bigr)}{\Phi^{-1}(1-u)/\rho_0} \\
    &= \bigl(1 + o(1)\bigr) (2\pi)^{-1/2} \rho_0 (-2\log u)^{-1/2} u^{(1 + o(1))/\rho_0^2}.
\end{align*}
If $\rho_0 \in (0, 1)$ is chosen such that $1/\rho_0^2 < 1+\eps$, then $\Pr[X > \Phi^{-1}(1-u)/\rho_0]$ is indeed of larger order than $u^{1+\eps}$,
as required.
\end{proof}

\end{appendix}

\section*{References}
\bibliographystyle{elsarticle-num}
\bibliography{biblio}

\begin{landscape}
\begin{table}[t]\caption{Finite level and asymptotic tail dependence of static and dynamic \textit{t}-copulas}\label{T:res}
    \begin{center}
    \begin{tabular}{lccccc @{\hspace{1cm}} ccccc}
\hline\hline
          &   \multicolumn{5}{l}{Static correlations}  &   \multicolumn{5}{l}{Dynamic correlations}   \\
Data&$\hat{\nu}$&$\lambda_{\text{year}}$&$\lambda_{\text{dec}}$&$\lambda_{\text{cent}}$&$\lambda$&$\hat{\nu}|\hat{\rho}_t$&$\lambda_{\text{year}}$&$\lambda_{\text{dec}}$&$\lambda_{\text{cent}}$&$\lambda$ \\
\hline
    Daily stock market &       &  &  &  &       &       &    &   &     &  \\
    DJ-NQ & 6.06  & 0.36  & 0.33  & 0.32  & 0.31  & $\infty$ & 0.26  & 0.18  & 0.13  & 0.00 \\
    Fra-Ger & 2.89  & 0.51  & 0.50  & 0.50  & 0.50  & 17.51 & 0.30  & 0.25  & 0.22  & 0.21 \\
    \hline
    Monthly stock market &       &       &       &       &       &       &       &       &       &  \\
    Ger-Jap & 12.39 & 0.24  & 0.11  & 0.05  & 0.02  & $\infty$ & 0.23  & 0.08  & 0.04  & 0.00 \\
    Ger-UK & 6.22  & 0.36  & 0.25  & 0.19  & 0.18  & 15.95 & 0.37  & 0.37  & 0.36  & 0.08 \\
    Ger-US & 3.93  & 0.39  & 0.31  & 0.28  & 0.27  & 11.20 & 0.35  & 0.23  & 0.15  & 0.11 \\
    Jap-UK & 10.32 & 0.26  & 0.13  & 0.07  & 0.04  & 41.03 & 0.25  & 0.10  & 0.03  & 0.00 \\
    Jap-US & 9.91  & 0.26  & 0.13  & 0.07  & 0.05  & $\infty$ & 0.26  & 0.11  & 0.06  & 0.00 \\
    UK-US & 14.85 & 0.41  & 0.25  & 0.14  & 0.07  & $\infty$ & 0.47  & 0.43  & 0.42  & 0.00 \\
    \hline
    Exchange rates &       &       &       &       &       &       &       &       &       &  \\
    EUR-GBP & 8.03  & 0.37  & 0.33  & 0.31  & 0.28  & 39.62 & 0.31  & 0.23  & 0.18  & 0.03 \\
    EUR-JPY & 5.51  & 0.20  & 0.18  & 0.17  & 0.16  & $\infty$ & 0.14  & 0.09  & 0.06  & 0.00 \\
    GBP-JPY & 4.55  & 0.17  & 0.16  & 0.15  & 0.15  & $\infty$ & 0.21  & 0.16  & 0.13  & 0.00 \\

\hline \hline \multicolumn{11}{p{17cm}}{\footnotesize 
\textbf{Note.} This table reports on the estimates of the degrees-of-freedom parameter $\nu$ of a
\textit{t}-copula conditional on constant correlation (column 2) and on time-varying
correlations driven by a SCAR process (column 7). The remaining
columns show measures of tail dependence at finite and asymptotic
quantiles. The margins are fitted via Gaussian stochastic
volatility models.}
\end{tabular}
\end{center}
\end{table}
\end{landscape}

\end{document}